\documentclass[12pt]{article}

\usepackage{latexsym}
\usepackage{epsfig,graphics}
\usepackage{amsmath,amssymb,amsthm}
\usepackage{graphics,epsfig,calc}
\usepackage{upref}
 \newcommand{\beqn}{\begin{eqnarray}}
 \newcommand{\eeqn}{\end{eqnarray}}
 \newcommand{\be}{\begin{equation}}
 \newcommand{\ee}{\end{equation}}
 \newcommand{\ba}{\begin{array}}
 \newcommand{\ea}{\end{array}}

 \newcommand{\pa}{\partial}
 \newcommand{\re}{\ref}
 \newcommand{\ci}{\cite}
 \newcommand{\ds}{\displaystyle}
 \newcommand{\la}{\label}

 \newcommand{\fr}{\frac}

\newcommand{\ov}{\overline}

\newcommand{\ti}{\tilde}

\newcommand{\cF}{{\cal F}}
\newcommand{\ve}{\varepsilon}

\newcommand{\de}{\delta}

\newcommand{\cH}{{\cal H}}

\newcommand{\cK}{{\cal K}}
\newcommand{\cL}{{\cal L}}

\newcommand{\cM}{{\cal M}}
\newcommand{\cO}{{\cal O}}

\newcommand{\cR}{{\cal R}}

\newcommand{\cV}{{\cal V}}
\newcommand{\bI}{{\bf I}}

\newcommand{\al}{\alpha}
\newcommand{\ga}{\gamma}
\newcommand{\Ga}{\Gamma}
\newcommand{\si}{\sigma}
\newcommand{\om}{\omega}
\newcommand{\vk}{\varkappa}

\newcommand{\na}{\nabla}

\newcommand{\lam}{\lambda}

\newcommand{\R}{\mathbb{R}}

\newcommand{\C}{\mathbb{C}}
\newcommand{\N}{\mathbb{N}}

\renewcommand{\theequation}{\thesection.\arabic{equation}}
\renewcommand{\thesection}{\arabic{section}}
\renewcommand{\thesubsection}{\arabic{section}.\arabic{subsection}}
\newtheorem{theorem}{Theorem}[section]

\renewcommand{\thetheorem}{\arabic{section}.\arabic{theorem}}
\newtheorem{defin}[theorem]{Definition}

\newtheorem{lemma}[theorem]{Lemma}
\newtheorem{remark}[theorem]{Remark}

\newtheorem{cor}[theorem]{Corollary}
\newtheorem{pro}[theorem]{Proposition}

\newcommand{\bd}{\begin{defin}}
 \newcommand{\ed}{\end{defin}}
\newcommand{\bt}{\begin{theorem}}
 \newcommand{\et}{\end{theorem}}
\newcommand{\bp}{\begin{pro}}
 \newcommand{\ep}{\end{pro}}
\newcommand{\bl}{\begin{lemma}}
\newcommand{\el}{\end{lemma}}
\newcommand{\bc}{\begin{cor}}
\newcommand{\ec}{\end{cor}}
\newcommand{\br}{\begin{remark} }
\newcommand{\er}{\end{remark}}

\begin{document}
\begin{titlepage}

\begin{center}
{\Large\bf
Weighted energy decay for \bigskip\\
magnetic Klein-Gordon equation} \vspace{1cm}
\\
{\large A.~I.~Komech}
\footnote{
Supported partly
 by the Alexander von Humboldt
Research Award and by the Austrian Science Fund (FWF): P22198-N13.}
\\
{\it Faculty of  Mathematics Vienna University\\
and Institute for Information Transmission Problems RAS}\\
 e-mail:~alexander.komech@univie.ac.at
\medskip\\

{\large E.~A.~Kopylova}
\footnote{Supported partly by the 
Austrian Science Fund (FWF): M1329-N13,
and RFBR grants.}\\
{\it Faculty of  Mathematics Vienna University\\
and Institute for Information Transmission Problems RAS}\\
e-mail:~elena.kopylova@univie.ac.at
\end{center}

\date{}

\vspace{0.5cm}
\begin{abstract}
\noindent We obtain a dispersive long-time decay in weighted
energy norms for solutions of  3D Klein-Gordon equation
with magnetic and scalar potentials.
The decay extends the results
obtained by Jensen and Kato for the  Schr\"odinger equation
with  scalar potential. For the proof we develop the spectral
theory of Agmon, Jensen and Kato and minimal escape velocities 
estimates of  Hunziker, Sigal and Soffer.
\smallskip

\smallskip

\end{abstract}

\end{titlepage}

\setcounter{equation}{0}
\section{Introduction}
We establish a dispersive long time decay
for  solutions to 3D  magnetic Klein-Gordon  equation
\be\la{KGE}
  \ddot\psi(x,t)=\big(\nabla-iA(x)\big)^2\psi(x,t)-m^2\psi(x,t)-
  V(x)\psi(x,t),\quad m>0.
\ee
For $s,\si\in\R$,
denote by $\cH^s_\si=\cH^s_\si (\R^3)$
the weighted Sobolev spaces introduced by Agmon, \ci{A},
with the finite norms
\be\la{norm}
  \Vert\psi\Vert_{\cH^s_\si}=\Vert\langle x
\rangle^\si\langle\na\rangle^s\psi\Vert_{L^2(\R^3)}<\infty,
\quad\quad \langle x\rangle=(1+|x|^2)^{1/2}.
\ee
\hspace{-0.15 cm} We assume that $V(x)\in C^1(\R^3)$, $A_j\in C^4(\R^3)$ are
real  functions, and for some $\beta>3$  the bounds hold
\be \label{V}
  |V(x)|+|\na V(x)|+\sum\limits_{|\al|\le 4}\sum\limits_{j=1}^3
  |D^\al A_j(x)|\le C\langle x\rangle^{-\beta},
  \quad x\in\R^3.
\ee
We restrict ourselves to the ``regular case'' in the terminology of
\cite{jeka} where the truncated resolvent of the corresponding  magnetic 
Schr\"odinger operator  
$$
H=(i\nabla+A\big)^2+V=-\Delta+2iA\cdot\na+i\nabla\cdot A+A^2+V
$$  
is bounded at the edge point
$0$  of the continuous spectrum. In other words, the point
$0$ is neither an eigenvalue nor a resonance for the operator
$H$; this holds for {\it generic potentials}.
\smallskip

In vector form,
equation (\ref{KGE}) reads
\be\la{KGEr}
  i\dot \Psi(t)=\cK\Psi(t),
\ee
where
$$
 \Psi(t)=\left(\begin{array}{c}
  \psi(t)
  \\
  \dot\psi(t)
  \end{array}\right),
~~~~~~~
\cK =\left(\begin{array}{cc}
  0               &   i
  \\
  i\big((\nabla-iA(x))^2-m^2-V(x)\big)   &   0
  \end{array}\right).
$$
Denote $\cF_{\si}=\cH^1_\si\oplus \cH^0_\si$,
and let $U(t):\cF_{0}\to\cF_{0}$ be the dynamical group of equation (\re{KGEr}).

Our main result is the following long time decay: in the regular case
for any $\sigma>5/2$ 
\be\label{full}
 \Vert U(t)\Psi(0)\Vert_{\cF_{-\si}}
 \le C (1+t)^{-3/2}\Vert\Psi(0)\Vert_{\cF_{\si}},\quad t>0
\ee
for  solutions to (\re{KGEr})
with  initial data $\Psi(0)$ from the space of continuous spectrum of $\cK$.
\smallskip

Let us comment on previous results in this direction.
The decay of type (\re{full}) in weighted norms has been established first
by Jensen and Kato \ci{jeka} for  3D Schr\"odinger equation 
with scalar potential. 
The result has been extended to more general PDEs of the Schr\"odinger type 
by Murata \ci{M}. The survey of the results can be found in \ci{Schlag}.
For the  Klein-Gordon and  wave equations with scalar potential, 
the weighed energy decay has been established in \ci{3Dkg} and \ci{wave},
and for the Dirac equation in \cite{Bo}.
The Strichartz estimates for magnetic Schr\"odinger, wave,
Klein-Gordon and Dirac equations with smallness conditions on the
potentials were obtained in \ci{AF2008, AF2010, GST}
and for magnetic Schr\"odinger equations with large potentials in \ci{EGS}.

The decay in  weighted  norms for magnetic Schr\"odinger equation 
has been established  in \ci{magS}.
For the  Klein-Gordon equations with magnetic potential, the decay
$\sim t^{-3/2}$ was obtained  by Vainberg \ci{V76, V89}
in local energy norms for initial data with a compact support.
However, the decay in weighed energy norms
for  magnetic Klein-Gordon equation was not obtain up to now.

Let us comment on our approach. We extend  the method of 
Jensen and Kato \ci{jeka}
to the Klein-Gordon  equation with  magnetic potential.
The main problem  consists in
the presence of the first order derivatives 
in the perturbation. 
These derivatives cannot be handled with the perturbation theory
like \ci{3Dkg} since the corresponding terms do not decay 
in suitable norms.
\smallskip

Our main novelties are Propositions \re{pr1} and \re{pr2} on
decay of propagators far from thresholds. First, we prove
the decay for magnetic Klein-Gordon equation with $V=0$.
The proof rely on the Mourre estimates for the operator
$B_A=((i\nabla+A\big)^2+m^2)^{1/2}$ and the
minimal escape velocity estimates  of Hunziker, Sigal and Soffer
\cite{HSS99} and their development by Boussaid \cite{Bo}.
Finally, we obtain the decay for the Klein-Gordon equation with $V\not=0$
using the Born perturbation series
and our recent results on the decay of the
magnetic Schr\"odinger resolvent \ci{magS}.

\setcounter{equation}{0}
\section{Spectral properties}\la{sp}
Denote $L^2=L^2(\R^3)$.
Similarly to \ci[p. 589]{jeka}, \ci[formula (3.1)]{M}
and \ci[\S 3.2]{magS},
let us introduce a generalized eigenspace $\cM$
for  the  Schr\"odinger operator 
$H$:
$$
{\cM}=\{\psi\in \cH^0_{-1/2-0}:\,~(1+A_0W)\psi=0\},
$$
where $A_0$ is the operator with the integral kernel $1/4\pi|x-y|$
and $W=2iA\cdot\na+i\nabla\cdot A+A^2+V$.
Functions $\psi\in \cM\cap L^2$ are the zero eigenfunctions of
$H$ and 
functions $\psi\in \cM\setminus L^2$ are  the {\bf zero
resonances} of  $H$.
\smallskip

Our key assumption is the following spectral condition (cf. Condition (i) in \ci[Theorem 7.2]{M}):
\be\la{SC}
\cM=0
\ee
In  other words, the point zero is neither an eigenvalue nor a resonance for
the operator $H$. Condition (\re{SC})
holds for {\it a generic} $W$.
Denote by $\cL (B_1,B_2)$ the Banach space of bounded linear operators
from a Banach space $B_1$ to a Banach space $B_2$.
Denote by $R(\om)=(H-\om)^{-1}$
the resolvent of the operator  $H$.
Let us  collect the properties of $R(\om)$ obtained in \cite{magS}
under conditions  (\ref{V}) and (\ref{SC}):
\begin{lemma}\la{sp1}
Let condition (\re{V})  holds. 
Then\\
i) For $\om>0$, the limiting absorption principle holds
\be\la{lapp}
R(\om\pm i\ve)\to R(\om\pm i0),\quad \ve\to 0+
\ee
in  $\cL(\cH^0_{\si}, \cH^2_{-\si})$ with $\si>1/2$.
\medskip\\
ii) For $k=0,1,2$,
$\si>1/2+k$,  and $s=0,1$, the asymptotics hold 
\be\la{H}
 \Vert R^{(k)}(\om)\Vert_{{\cal L}(\cH^s_\si;\cH^{s+l}_{-\si})}
 ={\cal O}(|\om|^{-\fr{1-l+k}2}),
  \quad |\om|\to\infty,\quad\om\in\C\setminus[0,\infty).
\ee
where $l=0,1$ for $s=0$, and $l=0,-1$ for $s=1$.
\end{lemma}
Note that asymptotics (\re{H}) have been proved in \ci{magS} 
for $s=0$ only. In the case $s=1$ the proof is given  in Appendix.
\begin{lemma}\la{sp2}
Let conditions (\re{V}) and (\re{SC}) hold. Then \\
i) For $\si>1$  the  asymptotics hold
\be\la{lappexp}
  \Vert R(\om)-R(0)\Vert_{\cL(\cH^0_{\si},\cH^2_{-\si})}
  \to 0,\quad\om\to 0,\quad\om\in \C\setminus[0,\infty)
\ee
where  $R(0):\cH^0_{\si}\to \cH^2_{-\si'}$ is a continuous operator.
\medskip\\
ii) For $k=1,2$ and $\si>1/2+k$ the asymptotics hold
\be\la{Zca}
\Vert R^{(k)}(\om)\Vert_{\cL(\cH^0_{\si},\cH^2_{-\si})}
=\cO(|\om|^{1/2-k}),~~\om\!\to 0,~~\om\in \C\setminus[0,\infty).
\ee
\end{lemma}
Denote $\Ga:=(-\infty,-m)\cup(m,\infty)$
and let  $\cR(\om)=(\cK-\om)^{-1}$ be
the resolvent of the operator $\cK$. 
The resolvent $\cR$ can be
expressed in terms of the resolvent $R$:
\begin{equation}\label{RKG}
   \cR(\omega)=
   \left(\begin{array}{cc}
      \omega R(\omega^2-m^2)           &   iR(\omega^2-m^2)
   \\
  -i(1 +\omega^2 R(\omega^2-m^2))      &   \omega R(\omega^2-m^2)
     \end{array}
   \right)
\end{equation}
Hence, the properties of $R$  imply the corresponding
properties of $\cR$:
\begin{lemma}\la{SP}
Let conditions (\re{V}) and (\re{SC}) hold. Then \\
i) The limiting absorption principle holds:
\be\la{lap1}
 \cR(\om\pm i\ve)\to \cR(\om\pm i0),\quad\om\in\Ga\quad\ve\to 0+
\ee
in  $\cL (\cF_\si,\cF_{-\si})$ with $\si>1/2$.
\medskip\\
ii) For  $\om\in\C\setminus\ov\Gamma$ the  asymptotics hold
\beqn\la{expRKG}
\!\!\!\Vert \cR(\om)\Vert_{\cL (\cF_{\si},\cF_{-\si})}
\!\!&=&\!\!\cO(1),\quad \om\pm m\to 0,\quad \si>1\\
\la{R0dif}
\!\!\!\Vert \cR^{(k)}(\om)\Vert_{\cL (\cF_{\si},\cF_{-\si})}
\!\!&=&\!\!\cO(|\om\pm m|^{1/2-k}),\quad\om\pm m\to 0,\quad \si>1/2+k,
\quad k=1,~2,...
\eeqn
iii) For $k=0,1,2,...$ and  $\si>1/2+k$ the asymptotics hold
  \begin{equation}\label{bR0}
    \Vert \cR^{(k)}(\om) \Vert_{\cL (\cF_{\sigma},\cF_{-\sigma})}=\cO(1)
    \quad\om\to\infty,\quad\om\in\C\setminus\Ga
  \end{equation}
\end{lemma}
Under  conditions (\re{V}) and  (\re{SC})
the representation holds
\be\la{srelap}
U(t)P_c(\cK)\Psi(0)=\fr{1}{2\pi i}\int_\Gamma
e^{-i\om t}[\cR(\om+i0)-\cR(\om-i0)]\Psi(0)~d\om,~~~t\in\R
\ee
for initial state $\Psi(0)\in\cF_\si$ with $\si>1$.
Here
$$
P_c(\cK)=\bI_{\ov\Ga}(\cK)
$$
is the projector associated with the continuous spectrum of $\cK$.
The representation (\re{srelap}) follows from  the Cauchy residue theorem,
and  Lemma \re{SP} (cf. \ci[\S 2.2]{3Dkg}).
\setcounter{equation}{0}
\section{ Time decay}
We are now able to state our main result
\begin{theorem}\la{wed} (Weighed energy decay)
Let assumptions  (\re{V}) and  (\re{SC}) hold.
Then for $\si>5/2$ the time decay holds
\be\la{WD}
\Vert U(t)P_c(\cK)\Vert_{\cL (\cF_\si,\cF_{-\si})}
\le C\langle t\rangle^{-3/2},\quad t\in\R.
\ee
\end{theorem}
We prove the decay separately for the components of the solution
near thresholds and far from thresholds. More precisely, we choose a function 
$\chi_m\in C_0^{\infty}(\R)$   supported in
a sufficiently small neighborhood of $[-m,m]$. 
Then
\be\la{WDF}
\Vert U(t)P_c(\cK)\Vert_{\cL (\cF_\si,\cF_{-\si})}
\le\Vert U(t)\chi_m(\cK)P_c(K)\Vert_{\cL (\cF_\si,\cF_{-\si})}
+\Vert U(t)(1-\chi_m)(\cK)\Vert_{\cL (\cF_\si,\cF_{-\si})}.
\ee
The decay of  the first {\it low energy  component} can be treated by 
the method of  Jensen and Kato \cite{jeka}. Namely, using the spectral 
representation (cf. (\re{srelap}))
\be\la{srelap1}
U(t)\chi_m(\cK)P_c(\cK)\Psi(0)=\fr{1}{2\pi i}\int_\Gamma
e^{-i\om t}\chi_m(\om)[\cR(\om+i0)-\cR(\om-i0)]\Psi(0)~d\om,~~~t\in\R,
\ee
and   asymptotics (\re{expRKG}) - (\re{R0dif}), we obtain for $\si>5/2$
\be\la{WD1}
\Vert U(t)\chi_m(\cK)P_c(\cK)\Vert_{\cL (\cF_\si,\cF_{-\si})}
\le C(\si)\langle t\rangle^{-3/2},\quad t\in\R
\ee
by \ci[Lemma 10.2]{jeka}.
To treat the decay of the second {\it high energy component}
we cannot use the spectral representation since the resolvent $\cR(\om)$
does not decay in $\cL (\cF_\si,\cF_{-\si})$ as $\om\to\infty$
(see (\ref{bR0})).
We obtain the required decay in the following way.
First, we consider the
Klein-Gordon equation  without a scalar 
potential, i.e with $V= 0$:
\be\la{KGEA}
i\dot\Psi(t)=\cK_0\Psi(t),
~~~~~~~
\cK_0=\left(\begin{array}{cc}
  0               &   i
  \\
  i\big((\nabla-iA(x))^2-m^2\big)   &   0
  \end{array}\right).
\ee
Denote $U_0(t):\cF_0\to\cF_0$ the dynamical group of equation (\re{KGEA}).
Applying  the minimal escape velocity estimates  of Hunziker, Sigal and Soffer
\cite[Theorem 1.1.]{HSS99} and their modification  
\cite[Proposition 2.2]{Bo} we will prove
\begin{pro}\la{pr1} (The case $V= 0$)
Let assumption  (\re{V})  hold.
Then for any bounded $\chi\in C^\infty(\R)$
supported in $\Ga$, any $\si\ge 2$ and  any $\ve>0$
the decay holds
\be\la{WD2}
\Vert U_0(t)\chi(\cK_0)\Vert_{\cL (\cF_\si,\cF_{-\si})}
\le C(\ve)\langle t\rangle^{-2+\ve},\quad t\in\R.
\ee
\end{pro}
Finally, we will prove the decay for the Klein-Gordon 
equation with $V\not =0$ using 
the Born perturbation series (\re{id}).
\begin{pro}\la{pr2}(The case $V\not = 0$)
Let assumption  (\re{V}) hold.
Then for any bounded $\chi\in C^\infty(\R)$ supported in $\Ga$,
any  $\si>5/2$ and any $\ve>0$ the decay holds
\be\la{WD3}
\Vert U(t)\chi(\cK)\Vert_{\cL (\cF_\si,\cF_{-\si})}
\le C(\si,\ve)\langle t\rangle^{-2+\ve},\quad t\in\R.
\ee
\end{pro}
Theorem \re{wed} follows from (\re{WDF}),
(\re{WD1}), and Propositions \re{pr1} and \re{pr2}.
We prove Propositions \re{pr1}  and \re{pr2} in the 
remaining part of the paper.
\setcounter{equation}{0}
\section{The case $V= 0$ }
First, we prove   Proposition  \re{pr1}. Denote
$$
B=\Big[(i\nabla-A(x)\big)^2+m^2\Big]^{1/2}
$$
which  is  positive and self-adjoint in $L^2$.
Then
\be\la{HB-rep}
U_0(t)=\left(\begin{array}{cc}
  \cos\, Bt           &   B^{-1}\sin\, Bt
  \\\\
  -B\sin\, Bt       &   \cos\, Bt
  \end{array}\right).
\ee
Hence, for the proof of (\re{WD2}) it suffices to check that
\be\la{bs1}
\Vert e^{-itB}\chi(B)\Vert_{{\cal L}(H^0_{\si}, H^0_{-\si})}
\le C(\ve)\langle t\rangle^{-2+\ve},\quad t\in\R
\ee
for $\si\ge 2$, $\ve>0$,  and any bounded $\chi\in C^\infty(\R)$
 with support in $(m,\infty)$.
We will deduce (\re{bs1}) from the minimal escape velocity estimates
\cite{HSS99} 
which rely on the Mourre estimates for the operator $B$.
\subsection{ Mourre estimates}
Denote
\be\la{PPB}
P=\fr i2(x\cdot\nabla+\nabla\cdot x),\quad P_B=PB^{-1}+B^{-1}P.
\ee
and let $\bI_M$ be the characteristic function of a set $M$.
\begin{lemma}\la{Mest}
Suppose that assumption  (\re{V})  holds. Then \\
i) For any $\theta\in (0,1)$
there exists $\nu\ge 0$ such that 
\be\la{ME}
 \bI_{|B|\ge m+\nu}\,i[B,P_B]\,\bI_{|B|\ge m+\nu}
 \ge\theta\,\bI_{|B|\ge m+\nu}.
\ee
ii) For any $\lam\in\Ga$ and  any $\de>0$, there exists $\mu>0$
such that 
\be\la{ME1}
\bI_{|B-\lam|\le\mu}\,i[B,P_B]\,\bI_{|B-\lam|\le\mu}
\ge \Big(\fr{\lam^2-m^2}{\lam^2}-\de\Big)\,\bI_{|B-\lam|\le\mu}.
\ee
\end{lemma}
\begin{proof}
{\it Step i)}
Let us obtain a suitable formula for commutator $[B,P_B]$.
First,
$$
[B,P_B]=[B,P]B^{-1}+B^{-1}[B,P].
$$
Further, we express $[B,P]$ via $[B^2,P]$ following \ci{SW}.
Namely, using the Kato square root formula \cite[page 317]{RS1}, we
obtain for any $\psi\in L^2$ 
\be\la{sqr} B\psi=\fr
1{\pi}\int_0^\infty\om^{-1/2}B^2(B^2+\om)^{-1}\psi\,d\om. 
\ee
Hence, one has
\be\la{sqr1} [B,P]=\fr
1{\pi}\int_0^\infty\om^{-1/2}[B^2(B^2+\om)^{-1},P]\,d\om. 
\ee
Further,
$$
(B^2+\om)[B^2(B^2+\om)^{-1},P](B^2+\om)=
B^2P(B^2+\om)-(B^2+\om)PB^2=\om[B^2,P].
$$
Then (\re {sqr}) becomes
\be\la{sqr2}
[B,P]=\fr 1{\pi}\int_0^\infty\om^{1/2}
(B^2+\om)^{-1}[B^2,P](B^2+\om)^{-1}\,d\om.
\ee
It is easy to  calculate
\be\la{a1}
i[B^2,P]=B^2-m^2+Q,
\ee
where
$$
Q=-A^2+2i(\nabla\cdot A)-x\cdot(\nabla A^2)
+2ix\cdot(\nabla(\nabla\cdot A))
+i\sum\limits_{j,k}x_j(\nabla_j A_k)\nabla_k.
$$
Substituting into (\re{sqr2}), we get
\beqn\la{sqr3}
\left.
\ba{rl}
i[B,P]B^{-1}&=\ds\fr{1}{\pi}\int_0^\infty\om^{1/2}
(B^2-m^2)B^{-1}(B^2+\om)^{-2}\,d\om\\\\
&+\ds\fr{1}{\pi}\int_0^\infty\om^{1/2}
(B^2+\om)^{-1}QB^{-1}(B^2+\om)^{-1}\,d\om=J_1+J_2\\\\\\
iB^{-1}[B,P]&=\ds\fr{1}{\pi}\int_0^\infty\om^{1/2}
(B^2-m^2)B^{-1}(B^2+\om)^{-2}\,d\om\\\\
&+\ds\fr{1}{\pi}\int_0^\infty\om^{1/2}
(B^2+\om)^{-1}B^{-1}Q(B^2+\om)^{-1}\,d\om=J_1+J_3
\ea
\right|
\eeqn

Applying the  integration by parts, we rewrite  $J_1$  as
\beqn\nonumber
J_1&=&-\fr{1}{\pi}\int_0^\infty\om^{1/2}
\fr{d}{d\om}(B^2+\om)^{-1}(B^2-m^2)B^{-1}\,d\om\\
\la{sqr4}\\
\nonumber
&=&\fr 12\int_0^\infty\om^{-1/2}(B^2+\om)^{-1}(B^2-m^2)B^{-1}\,d\om
=\fr 12(B^2-m^2)B^{-2},
\eeqn
which follows from (\re{sqr}) and the  bound
\be\la{Bm}
\Vert(B^2+\om)^{-1}\psi\Vert_{L^2}
\le(m^2+\om)^{-1}\Vert\psi\Vert_{L^2}.
\ee
Finally,
\be\la{k1}
i[B,P_B]=\fr{B^2-m^2}{B^2}+J,\quad J=J_2+J_3. 
\ee
\smallskip\\
{\it Step ii)}
Let us prove that $J=J_2+J_3:L^2\to L^2$ is a  compact operator.
First, bounds (\re{V}) and (\re{Bm}) imply for any $0<\si<\beta$
and $0<\al<1$
\beqn\nonumber
\Vert (B^2+\om)^{-1+\al}QB^{-1}(B^2+\om)^{-1}\psi\Vert_{\cH^0_\si}
&\le& C(1+\om)^{-2+\al}\Vert\psi\Vert_{L^2},\\
\nonumber\\
\nonumber
\Vert (B^2+\om)^{-1+\al}B^{-1}Q(B^2+\om)^{-1}\psi\Vert_{\cH^0_\si}
&\le& C(1+\om)^{-2+\al}\Vert\psi\Vert_{L^2}
\eeqn
by the technique of \ci{Seeley} and the standard technique
of PDOs \ci{ABG,Sh,T}.
Second, for any $0<\al<1$ and $\phi\in\cH^0_\si$ the bound holds
\be\la{sqr5}
\Vert(B^2+\om)^{-\al}\phi\Vert_{\cH^{2\al}_\si}
\le C\Vert\phi\Vert_{\cH^0_\si}.
\ee
Indeed, using the technique \ci{Seeley}, 
we get
$$
\Vert(B^2+\om)^{-\al}\phi\Vert_{\cH^{2\al}_\si}\le C
\Vert B^{2\al}(B^2+\om)^{-\al}\phi\Vert_{\cH^{0}_\si}
\le C_1\Vert\phi\Vert_{\cH^{0}_\si}
$$
since $B$ is a positive elliptic first order PDO.
Finally, choosing  $0<\al<1/2$, we obtain
\be\la{sqr6}
\Vert J_2\psi\Vert_{\cH^{2\al}_\si}+\Vert J_3\psi\Vert_{\cH^{2\al}_\si}
\le C\Vert\psi\Vert_{L^2}.
\ee
Thetefore, $J_2,~J_3:L^2\to L^2$ are  compact operators 
since the embedding $\cH^{2\al}_\si\subset L^2$ is compact
by Sobolev's Embedding Theorem.
\smallskip\\
{\it Step iii)}
In the case $A=0$ (and then $J=0$) bounds (\re{ME}) and  (\re{ME1})
follow from (\re{k1}).
For $A\not=0$ and any $\vk>0$ we split the compact operator $J$ as
$$
J=J_{\vk}+\sum\limits_1^N|f_j\rangle\langle g_j|,
$$
where $\Vert J_{\vk}\Vert\le\vk$, and $f_j, g_j\in L^2$.
Then
$$
\Vert\,\bI_{|B-\lam|\le\mu}|\,f_j\rangle
\langle g_j\,|\bI_{|B-\lam|\le\mu}\Vert_{L^2\to L^2}
\le \Vert\,\bI_{|B-\lam|\le\mu}\hat f_j\Vert_{L^2}\cdot
\Vert\,\bI_{|B-\lam|\le\mu}\hat g_j\Vert_{L^2}\to 0,\quad\mu\to 0
$$
due to absolute continuity of spectral representatives
$\hat f_j, \hat g_j\in L^2([0,\infty),X)$ of
$f_j$, $g_j$ in the spectral resolution of $B$, where $X$ is an appropriate
Hilbert space.
Hence, for sufficiently small $\vk$ and $\mu$ bound (\re{ME1}) follows.
Similarly,  bound (\re{ME}) follows for  sufficiently small $\vk$ and
sufficiently large $\nu$.
\end{proof}
\subsection{Minimal escape velocity}
Here we  adapt the methods of \cite[Theorem 1.1]{HSS99} to our case 
(see also \cite[Theorem 2.1]{Bo})
\begin{lemma}\la{mev}
Let assumption (\re{V}) hold. Then for any  bounded
$\chi\in C^{\infty}$ with support in $\Ga$, there exists $\theta>0$ such
that  for any $v\in (0,\theta)$,  any $a\in\R$, and any $\ve>0$
the bound holds
\be\la{bs}
\Vert\bI_{P_B\le a+v|t|}~e^{-itB}\chi(B)\bI_{P_B\ge a}\Vert
\le C(v,\ve)\langle t\rangle^{-2+\ve},\quad t\in\R,
\ee
where $C$ does not depend on $a$ and $t$.
\end{lemma}
\begin{proof}
According to \cite[Theorem 1.1]{HSS99} and  \cite[Theorem 2.1]{Bo}
bound (\re{bs}) follows from  the Mourre estimates (\re{ME}) - (\re{ME1})
and the boundedness of commutators
$ad_{P_B}^k(B):L^2\to L^2$ for $1\le k\le 3$, where
$$
ad_{P_B}^1(B)=[B,P_B],\quad {\rm and}~~
ad_{P_B}^k(B)=[ad_{P_B}^{k-1}(B),P_B].
$$
The boundedness of $ad_{P_B}^1(B)=[B,P_B]$ follows from
(\re{k1}) and (\re{sqr6}).\\
For $k=2,3$ we have 
$$
ad_{P_B}^{k}(B)=-i\big[ad_{P_B}^{k-1}(2J_1)+ad_{P_B}^{k-1}(J_2)+
ad_{P_B}^{k-1}(J_3)\big]
$$
by (\re{k1}).
The boundedness of $ad_{P_B}^{k-1}(J_2)$ and $ad_{P_B}^{k-1}(J_3)$ 
is obvious
due to (\re{V}) and definition (\re{sqr3}) of $J_2$ and $J_3$.
Hence, it remains to prove that
$[P_B,2J_1]$ and $[P_B,[P_B,2J_1]]$ are bounded in $L^2$.
The boundedness of $[B,P_B]$ imply the boundedness of
$$
[P_B,B^{-1}]=B^{-1}[B,P_B]B^{-1}.
$$
Then by (\re{sqr4}) the operator
$$
[P_B,2J_1]=[P_B,(B^2-m^2)B^{-2}]=-m^2[P_B,B^{-2}]
=-m^2\big([P_B,B^{-1}]B^{-1}+B^{-1}[P_B,B^{-1}]\big)
$$
is also bounded in $L^2$. Further, (\re{PPB}) and (\re{a1}) imply
\beqn\nonumber
[P_B,2J_1]&=&m^2B^{-2}[P_B,B^2]B^{-2}
=m^2B^{-3}[P,B^2]B^{-2}+m^2B^{-2}[P,B^2]B^{-3}\\
\nonumber
&=&m^2i\Big(2(B^2-m^2)B^{-5}+B^{-3}QB^{-2}+B^{-2}QB^{-3}\Big).
\eeqn
Hence, the boundedness of $[P_B,[P_B,2J_1]]$ in $L^2$ follows from (\re{V})
and  the  boundedness of  $[P_B,B^{-l}]$ for any $l\in\N$.
\end{proof}
\hspace{-6mm}{\bf  Proof  of  Proposition  \re{pr1}}
For any  $c\ge 0$ and any $\si> 0$ one has
$$
\langle P_B\rangle^{-\si}=\langle P_B\rangle^{-\si}\bI_{\pm P_B\le c|t|}
+{\cal O}(|t|^{-\si}),\quad |t|>1
$$
in $\cL(L^2, L^2)$. Hence,
$$
\langle P_B\rangle^{-\si}e^{-itB}\chi(B)\langle P_B\rangle^{-\si}
=\langle P_B\rangle^{-\si}\bI_{P_B\le (\theta-\ga)|t|/2}
e^{-itB}\chi(B)\bI_{P_B\ge -\theta |t|/2}\langle P_B\rangle^{-\si}
+{\cal O}(|t|^{-\si}),\quad\ga<\theta.
$$
Choosing $a=-\fr{\theta |t|}2$ and $v=\theta-\frac{\ga}2$ 
in Lemma \re{mev}, we obtain for  $\si= 2$ and $\ve>0$
$$
\Vert\langle P_B\rangle^{-\si}\langle x\rangle^{\si}\langle x\rangle^{-\si}
e^{-itB}\chi(B)\langle x\rangle^{-\si}\langle x\rangle^{\si}
\langle P_B\rangle^{-\si}\Vert_{\cL(L^2, L^2)}\le C(\ve)\langle t\rangle^{-2+\ve},
\quad t\in\R.
$$
Now (\re{bs1}) follows since
$\langle P_B\rangle^{-\si}\langle x\rangle^{\si}$
and $\langle x\rangle^{\si}\langle P_B\rangle^{-\si}$
are bounded  in $\cL(L^2, L^2)$.
This follows by the arguments from the proof of Proposition 2.2 
in \cite{Bo} (page 770), relying on the 
multi-commutator expansion \cite[Identity (B.24)]{HS00} and 
the identity \cite[(1.2)]{SS98}.

\setcounter{equation}{0}
\section{The case $V\not = 0$}
Here we  prove    Proposition  \re{pr2}.
Denote
\be\la{psih}
X(t):=U(t)\chi(\cK)\Psi(0)=
\fr 1{2\pi i}\int\limits_\Gamma\chi(\om)e^{-i\om t}
\Big[\cR(\om+i0)-\cR(\om-i0)\Big]\Psi_0~ d\om.
\ee
Our final goal is the bound
\be\la{Psi}
\Vert X(t)\Vert_{\cF _{-\si}}
 \le C(\si,\ve)\Vert \Psi_0\Vert_{\cF _\si}\langle t\rangle^{-2+\ve},
  \quad t\in\R,\quad\si>5/2.
\ee
Let us apply the Born perturbation series
\be\la{id}
  \cR(\om)
  = \cR_{0}(\omega)-\cR_{0}(\omega)\cV \cR_{0}(\omega)
  +\cR_{0}(\omega)\cV \cR_{0}(\omega)\cV \cR(\omega),
\ee
which follows by iteration of
$\cR(\om)= \cR_{0}(\omega)-\cR_{0}(\omega)\cV \cR(\omega)$.
Here
$\cR_{0}(\omega)=(\cK_0-\om)^{-1}$ is the resolvent of the operator $\cK_0$
and
\be\la{cV}
  \cV =\left(\begin{array}{cc}
  0               &   0
  \\
  -iV               &   0
\end{array}\right).
\ee 
Substituting (\re{id}) into (\re{psih}) we obtain
\beqn\nonumber
  X(t)
  &=&\fr 1{2\pi i}\int\limits_\Gamma\chi(\om) e^{-i\om t}
  \Big[\cR_{0}(\om+i0)-\cR_{0}(\om-i0)\Big]\Psi_0~ d\om\\
  \nonumber
  &+&\fr 1{2\pi i}\int\limits_\Gamma \chi(\om)e^{-i\om t}
  \Big[\cR_{0}(\om+i0)
  \cV \cR_{0}(\om+i0)-\cR_{0}(\om-i0)\cV \cR_{0}(\om-i0)\Big]\Psi_0~ d\om\\
  \la{XXX}
  &+&\frac 1{2\pi i}\int\limits_\Gamma \chi(\om)e^{-i\om t}
  \Big[\cR_{0}\cV \cR_{0}\cV \cR(\om+i0)
  -\cR_{0}\cV \cR_{0}\cV \cR(\om-i0)\Big]\Psi_0~ d\om\\
  &=&X_1(t)+X_2(t)+X_3(t),~~~~~~t\in\R \nonumber
\eeqn
We analyze each term $X_k$ separately.
\medskip\\
{\it Step i)}
For
$X_1(t)=U_0(t)\chi(\cK_0)\Psi(0)$ 
Proposition \ref{pr1} implies that for any $\si\ge 2$ and any $\ve>0$
\be\la{lins1}
  \Vert X_1(t)\Vert_{\cF _{-\si}}
 \le C(\ve)\Vert \Psi_0\Vert_{\cF _\si}\langle t\rangle^{-2+\ve},
  \quad t\in\R.
\ee
{\it Step ii)}
Consider the second term $X_2(t)$.
We can choose the function $\chi(\om)$ such that
$ \chi(\om)=\chi_1^2(\om)$.
Denote
$$
Y_1(t)=\fr 1{2\pi i}\int\limits_\Gamma \chi_1(\om) e^{-i\om t}
\Big[\cR_{0}(\om+i0)-\cR_{0}(\om-i0)\Big]\Psi_0~ d\om
$$
It is obvious that for $Y_1(t)$ the inequality (\ref{lins1}) also holds.
Namely,
\be\la{lins11}
  \Vert Y_1(t)\Vert_{\cF _{-\si}}
  \le C(\ve)\Vert \Psi_0\Vert_{\cF _\si}\langle t\rangle^{-2+\ve},
  \quad t\in\R,\quad \si\ge 2,\quad\ve>0.
\ee
Now the second term $X_2(t)$ can be rewritten as a convolution.
\begin{lemma}
The convolution representation holds
\be\la{P2}
  X_2(t)=
  i\int\limits_0^t U(t-\tau)\cV Y_1(\tau)~d\tau,~~~~t\in\R
\ee
where the integral converges in $\cF_{-\si}$ with $\si\ge 2$.
\end{lemma}
\begin{proof}
We have
\beqn\la{P22}
X_2(t)&=&\fr 1{2\pi i}\int\limits_\R e^{-i\om t}
\chi_1(\om)^2\cR_{0}(\om+i0)\cV  \cR_{0}(\om+i0)\Psi_0~ d\om\\
\nonumber
&-&\fr 1{2\pi i}\int\limits_\R e^{-i\om t}
\chi_1(\om)^2\cR_{0}(\om-i0)\cV  \cR_{0}(\om-i0)\Big]\Psi_0~ d\om
=X_2^+(t)+X_2^-(t)
\eeqn
Denote
$Y_{1}^+(t):=\theta(t)Y_{1}(t)$.
Then $\chi_1(\om) \cR_{0}(\om+i0)\Psi_0=i\ti Y_{1}^+(\om)$
and we obtain that
\beqn\nonumber
 X_2^+(t)&=&\fr 1{2\pi }\int\limits_\R e^{-i\om t}\chi_1(\om)
 \cR_{0}(\om+i0)\cV \ti Y_{1}^+(\om)~ d\om\\
\nonumber\\
\nonumber
&=&\fr 1{2\pi }\int\limits_\R e^{-i\om t}\chi_1(\om) \cR_{0}(\om+i0)
\cV \Big[\int_\R e^{i\om\tau}Y_{1}^+(\tau)d\tau\Big]d\om\\
\nonumber\\
\nonumber
&=&\fr 1{2\pi }(i\pa_t+i)^2\int\limits_\R\fr{e^{-i\om t}}{(\om+i)^2}
\chi_1(\om) \cR_{0}(\om+i0)\cV \Big[\int_\R e^{i\om\tau}
Y_{1}^+(\tau)d\tau\Big]d\om.
\eeqn
The last double integral  converges in $\cF_{-\si}$ with $\si\ge 2$
by (\re{lins11}) with $0<\ve<1$, Lemma \ref{SP} i), 
and (\re{bR0}) with $k=0$. 
Hence, we can change
the order of integration by the Fubini theorem and 
we obtain that
\be\la{p21}
X_{2}^+(t)=
\left\{\ba{cl}
\ds i\int_0^t U_0(t-\tau)\chi_1(\cK_0)\cV Y_{1}(\tau)d\tau&,~~t>0\\
0&,~~t<0
\ea
\right.
\ee
since
\beqn\nonumber
\fr 1{2\pi i}(i\pa_t+i)^2\int\limits_{\R}
\fr{e^{-i\om (t-\tau)}}{(\om+i)^2}\chi_1(\om) \cR_{0}(\om+i0)~d\om
&=&\fr 1{2\pi i}\int\limits_{\R}
e^{-i\om (t-\tau)}\chi_1(\om) \cR_{0}(\om+i0)~d\om\\
\nonumber
&=&\theta(t-\tau)U_0(t-\tau)\chi_1(\cK_0)
\eeqn
Similarly, we obtain
\be\la{p22}
X_{2}^-(t)=
\left\{\ba{cl}
0&,~~t>0\\
\ds i\int_0^tU_0(t-\tau)\chi_1(\cK_0)\cV Y_{1}(\tau)d\tau&,~~t<0
\ea
\right.
\ee
Now (\re{P2}) follows since $X_{2}(t)$ is the sum of two expressions
(\re{p21}) and (\re{p22}).
\end{proof}
Now we choose an arbitrary $\si\ge 2$, $0<\ve<1$ and 
$\si_1\in[2,~\min\{\si,\beta/2\})$. 
Applying   Proposition \ref{pr1} with $\chi_1$ instead $\chi$
 to the integrand in (\re{P2}), we obtain that
$$
  \Vert U_0(t-\tau)\chi_1(\cK_0)\cV Y_{1}(\tau)\Vert_{\cF _{-\si}}
  \le
  \Vert  U_0(t-\tau)\chi_1(\cK_0)\cV Y_{1}(\tau)\Vert_{\cF _{-\si_1}}
  \le\ds\fr{C\Vert\cV Y_{1}(\tau)\Vert_{\cF _{\si_1}}}{(1+|t-\tau|)^{2-\ve}}
$$
$$
  \le\ds\fr{C\Vert Y_{1}(\tau)\Vert_{\cF _{-\si_1}}}{(1+|t-\tau|)^{2-\ve}}
  \le\ds\fr{C\Vert \Psi_0\Vert_{\cF _{\si_1}}}
  {(1+|t-\tau|)^{2-\ve}(1+|\tau|)^{2-\ve}}
  \le\ds\fr{C\Vert \Psi_0\Vert_{\cF _\si}}
  {(1+|t-\tau|)^{2-\ve}(1+|\tau|)^{2-\ve}}
$$
Integrating, we obtain by (\re{P2}) that
\be\la{lins2}
  \Vert X_{2}(t)\Vert_{\cF _{-\si}}\le C(\ve)\Vert \Psi_0\Vert_{\cF _\si}
  \langle t\rangle^{-2+\ve}, \quad t\in\R,\quad \sigma\ge 2.
\ee
\medskip\\
{\it Step iii)}
Finally, we rewrite the last term in (\re{XXX}) as
\be\la{X3}
X_{3}(t)=\frac 1{2\pi i}\int\limits_{\Gamma}e^{-i\om t}\chi(\om)
N(\om)\Psi_0~ d\om,
\ee
where $N(\om):=M(\om+i0)-M(\om-i0)$ and
$$
M(\om):=\cR_{0}(\om)\cV \cR_{0}(\om)\cV \cR(\om)=
\cR_{0}L(\om)\cR(\om).
$$
First, we obtain the
asymptotics of $L(\om):= \cV \cR_{0}(\om)\cV$ for large $\om$.
\begin{lemma}\la{large1}
Let $\si>0$, $k=0,1,2$, and  $V$ satisfy (\ref{V}) with $\beta>1/2+k+\si$.
Then  the asymptotics hold
\begin{equation}\label{bRV}
  \Vert L^{(k)}(\om)\Vert_{\cL (\cF _{-\sigma},\cF _{\sigma})}
  ={\cal O}(|\om|^{-2}),\quad |\om|\to\infty,\quad \om\in\C\setminus\Ga.
  \end{equation}
\end{lemma}
\begin{proof}
Denote $R_0(\om)=(H_0-\om)^{-1}$, where $H_0$ corresponds to $H$ 
with $V= 0$, i.e., $H_0=(i\na+A)^2$.
Bounds (\ref{bRV})  follow from  the algebraic structure of the matrix
\be\la{ident}
  L^{(k)}(\om)=\cV \cR_{0}^{(k)}(\om)\cV =\left(\begin{array}{cc}
  0                                      &   0
  \\
  -iVR_{0}^{(k)}(\om^2-m^2)V            &   0
\end{array}\right)
\ee
For $\si>1/2+k$ asymptotics (\re{H}) with $s=1$ and $l=-1$ implies  that
$$ 
\Vert R^{(k)}_{0}(\om^2-m^2)\Vert_{{\cal L}(\cH^1_\si;\cH^{0}_{-\si})}
 ={\cal O}(|\om|^{-2}),
  \quad |\om|\to\infty,\quad\om\in\C\setminus\Ga,
\quad k=0,1,2.
$$
Therefore, for  $1/2+k<\beta-\si$ the asymptotics hold
$$
 \Vert VR^{(k)}_{0}(\om^2\!-m^2)V f\Vert_{\cH^{0}_{\si}}
 \le C\Vert R^{(k)}_{0}(\om^2\!-m^2)V f\Vert_{\cH^{0}_{\si-\beta}}
 = {\cal O}(|\om|^{-2})\Vert V f\Vert_{\cH^{1}_{\beta-\si}}
 ={\cal O}(|\om|^{-2})\Vert f\Vert_{\cH^{1}_{-\si}}.
$$
\end{proof}
Further, we obtain the asymptotics of $M(\om)$ 
and its derivatives for large $\om$.
\begin{lemma}\la{bM}
Let $V$  satisfy (\ref{V}) with $\beta>3$. Then
for $k=0,1,2$, the  asymptotics hold
 \begin{equation}\label{expbM}
  \Vert M^{(k)}(\om)\Vert_{\cL (\cF _{\sigma},\cF _{-\sigma})}=
{\cal O}(|\om|^{-2}),\quad |\om|\to\infty,~~\om\in\C\setminus\Ga,\quad
\si>1/2+k.
\end{equation}
\end{lemma}
\begin{proof}
The asymptotics (\ref{expbM}) follow from asymptotics (\ref{bR0})
for $\cR^{(k)}_{0}$ and  $\cR^{(k)}$, 
and asymptotics (\ref{bRV}) for $L^{(k)}$.
For example,  consider  the case $k=2$.
We have
\be\la{M2}
M''=\cR_{0}''L\cR+\cR_{0}L''\cR+\cR_{0}L\cR''
+2\cR_{0}'L'\cR+2\cR_{0}'L\cR'+2\cR_{0}L'\cR'.
\ee
For a fixed $\si>5/2$, let us choose
$\si'\in (5/2,\,\min\{\si,\beta-1/2\})$.
Then for the first term in (\ref{M2}) we obtain
by (\ref{bR0}) and (\ref{bRV})
\beqn\nonumber
\!\!\!&&\!\!\!\Vert \cR_{0}''(\om)L(\om)\cR_(\om) f\Vert_{\cF_{-\si}}
\le
\Vert \cR_{0}''(\om)L(\om)\cR(\om) f\Vert_{\cF_{-\si'}}\le
\Vert L(\om)\cR(\om) f\Vert_{\cF_{\si'}}
\\
\nonumber
\!\!\!&\le&\!\!\! \frac{C}{|\om|^{2}}
\Vert \cR(\om)f\Vert_{\cF_{-\si'}}
\le \frac{C_1}{|\om|^{2}}\Vert f\Vert_{\cF_{\si'}}
\le \frac{C_2}{|\om|^{2}}\Vert f\Vert_{\cF_{\si}},\quad\om\to\infty,\quad
\om\in\C\setminus\Ga.
\eeqn
Other  terms can be estimated similarly
choosing an appropriate values of $\si'$.
\end{proof} 
Now we prove the decay of $X_{3}(t)$.
By Lemma \ref{bM}
$$
(\chi N)''\in 
L^1(\Ga;\cL (\cF _{\si},\cF _{-\si}))
$$
with $\si>5/2$.
Hence, two times partial integration in (\re{X3}) implies that
$$
\Vert X_{3}(t)\Vert_{\cF _{-\si}}
 \le C(\si)\Vert \Psi_0\Vert_{\cF _\si}\langle t\rangle^{-2},
  \quad t\in\R
$$
Together with (\re{lins1}) and (\re{lins2})
this completes the proof of Proposition \ref{pr2}.
\appendix

\setcounter{section}{0}
\setcounter{equation}{0}
\protect\renewcommand{\thesection}{\Alph{section}}
\protect\renewcommand{\theequation}{\thesection.\arabic{equation}}
\protect\renewcommand{\thesubsection}{\thesection.\arabic{subsection}}
\protect\renewcommand{\thetheorem}{\Alph{section}.\arabic{theorem}}
\section{Decay of magnetic Schr\"odinger resolvent}
Here we prove  Lemma  \re{sp1} ii) for $s=1$.
First, we consider the case $V=0$.
Recall that
$$ 
H_0=(i\na+A)^2,\quad {\rm and}\quad R_{0}(\om)=(H_0-\om)^{-1}.
$$
\begin{lemma}\la{LA1}
Let $A(x)\in C^2(\R^3)$ be a  real function, and for some $\beta>2$
the bound holds
\be\la{nnA}
  |A(x)|+|\na A(x)|+|\na\na A(x)|\le C\langle x\rangle^{-\beta}.
\ee
Then for $l=-1,0,1$
and $\si>1/2$,  the asymptotics hold 
\be\la{AA}
 \Vert R_{0}(\om)\Vert_{{\cal L}(\cH^1_\si;\cH^{1+l}_{-\si})}
 ={\cal O}(|\om|^{-\fr{1-l}2}),
  \quad |\om|\to\infty,\quad\om\in\C\setminus[0,\infty).
\ee
\end{lemma}
\begin{proof}
{\it Step i)}
Consider  $l=0$. 
Applying the technique of PDO \cite{Sh, T}
we obtain for large $\om\in\C\setminus[0,\infty)$ 
\beqn\nonumber
\Vert R_{0}(\om)\psi\Vert_{\cH^{1}_{-\si}}
&\le&\Vert \nabla R_{0}(\om)\psi\Vert_{\cH^{0}_{-\si}}
+\Vert R_{0}(\om)\psi\Vert_{\cH^{0}_{-\si}}
\le C\Vert\sqrt{H_0+1}R_{0}(\om)\psi\Vert_{\cH^{0}_{-\si}}\\
\nonumber
&=&C_1\Vert R_{0}(\om)\sqrt{H_0+1}\,\psi\Vert_{\cH^{0}_{-\si}}
\le C_2|\om|^{-1/2}\Vert\sqrt{H_0+1}\,\psi\Vert_{\cH^{0}_{\si}}\\
\nonumber
&\le& C_3|\om|^{-1/2}\Vert\psi\Vert_{\cH^{1}_{\si}} 
\eeqn
by (\re{H}) with $k=s=l=0$ and $V=0$.
\smallskip\\
{\it Step ii)}
Similarly, (\re{H}) with $k=s=0$ and $l=1$,
implies for large $\om\in\C\setminus[0,\infty)$
\beqn\nonumber
\Vert R_{0}(\om)\psi\Vert_{\cH^{2}_{-\si}}
&=&\Vert (H_0+1)R_{0}(\om)\psi\Vert_{\cH^{0}_{-\si}}
\le C\Vert(\sqrt{-\Delta+1}\sqrt{H_0+1}\, 
R_{0}(\om)\psi\Vert_{\cH^{0}_{-\si}}
\\
\nonumber
&=&C\Vert\sqrt{-\Delta+1}\,R_{0}(\om)\sqrt{H_0+1}\,
\psi\Vert_{\cH^{0}_{-\si}} 
\le C_1\Vert R_{0}(\om)\sqrt{H_0+1}\,\psi\Vert_{\cH^{1}_{-\si}}\\
\nonumber
&\le& C_2\Vert\sqrt{H_0+1}\,\psi\Vert_{\cH^{0}_{\si}}
\le C\Vert \psi\Vert_{\cH^{1}_{\si}}
\eeqn
Then  (\re{AA}) with $l=1$ follows.
\smallskip\\
{\it Step iii)}
It remains to consider the case $l=-1$.
We have by (\re{AA}) with  $l=1$
\beqn\nonumber
\Vert R_{0}(\om)\psi\Vert_{\cH^{0}_{-\si}}
&=&\Vert \om^{-1}(-1+H_0 R_{0}(\om))\psi\Vert_{\cH^{0}_{-\si}}
\le C|\om|^{-1}\Big[\Vert\psi\Vert_{\cH^{0}_{-\si}}
+\Vert R_{0}(\om)\psi\Vert_{\cH^{2}_{-\si}}\Big]\\
\nonumber
&\le& C_1|\om|^{-1}\Vert\psi\Vert_{\cH^1_{-\si}}
\eeqn
\end{proof}
Now we consider  $V\not=0$.
\begin{lemma}\la{LA2}
Let for some $\beta>3$
$$
|V(x)|+|A(x)|+|\na A(x)|+|\na\na A(x)|
\le C\langle x\rangle^{-\beta}.
$$
Then for $k=0,1,2$, $\si>1/2+k$,  and $l=-1,0$, the asymptotics hold 
\be\la{AA1}
 \Vert R^{(k)}(\om)\Vert_{{\cal L}(\cH^1_\si;\cH^{1+l}_{-\si})}
 ={\cal O}(|\om|^{-\fr{1-l+k}2}),
  \quad |\om|\to\infty,\quad\om\in\C\setminus[0,\infty).
\ee
\end{lemma}
\begin{proof}
For $k=0$ asymptotics (\re{H}) follow from the  Born splitting
$$
R(\om)=R_{0}(\om)[1+VR_{0}(\om)]^{-1}
$$
and (\re{AA}), since the norm of the operator
$[1+VR_{0}(\om)]^{-1}: \cH^1_{\si}\to \cH^1_{\si}$ is bounded
for large $\om\in\C\setminus[0,\infty)$ and $\si\in(1/2,\beta/2]$.
\smallskip\\
For $k=1$ and $k=2$ we use the identities
\be\la{R1}
  R'=(1-RW)R_{\Delta}'(1-WR)
=R_{\Delta}'-RWR_{\Delta}'-R_{\Delta}'WR+RWR_{\Delta}'WR.
\ee
\beqn\la{R2}
   R''&=&(1-RW)R_{\Delta}''(1-WR)-2 R'WR_{\Delta}'(1-WR)\\
   \nonumber
   &=&R_{\Delta}''-RWR_{\Delta}''-R_{\Delta}''WR
+RWR_{\Delta}''WR-2 R'WR_{\Delta}'+2 R'WR_{\Delta}'WR.
\eeqn
and well-known asymptotics for $R_{\Delta}(\om)=(-\Delta-\om)^{-1}$ 
(see \cite{jeka, 3Dkg}):
\be\la{R3}
\Vert R_{\Delta}^{(k)}(\om)\Vert_{\cL(\cH^s_\si,\cH^{s+l}_{-\si})}
=\cO(|\om|^{-\fr {1-l+k}2}),\quad\om\to\infty,\quad
\om\in\C\setminus[0,\infty)
\ee
for $s\in\R$, $l=-1,0,1$, $k=0,1,2,...$ and $\si>k+1/2$.
Identities (\re{R1})-(\re{R2}) and asymptotics (\re{R3}) imply 
(\re{AA1}) (cf.  \cite[Theorem 3.8]{magS}).
\end{proof}

\end{document}